\documentclass[12pt,amsfonts]{article}
\usepackage{amsmath,amssymb,epsfig, amsthm, geometry, esint}
\usepackage{changebar}
\newtheorem{proposition}{Proposition}[section]

\newtheorem{corollary}[proposition]{Corollary}
\newtheorem{thm}[proposition]{Theorem}
\newtheorem{theorem}{Theorem}
\newtheorem{lemma}[proposition]{Lemma}
\newtheorem*{assumption}{(U) Uniform Constants}

\newtheorem{remark}[proposition]{Remark}
\newcommand{\ve}{\varepsilon}
\newcommand{\Si}{\mathcal{S}}

\newcommand{\C}{\mathcal{C}}

\newcommand{\ho}{\mathcal{H}}
\newcommand{\hofin}{\mathcal{H}_{\textrm{fin}}}
\newcommand{\Fb}{\mathcal{F}_b}

\newcommand{\G}{\mathcal{G}}
\newcommand{\I}{\mathcal{I}}

\newcommand{\nusrb}{\nu_{\mbox{\tiny SRB}}}

\newcommand{\Lp}{\mathcal{L}}
\newcommand{\pa}{\mathcal{P}}

\newcommand{\bDelta}{\overline{\Delta}}

\newcommand{\barF}{\overline{F}}
\newcommand{\dlj}{\Delta_{\ell,j}}
\newcommand{\pklj}{\mathcal{P}^k_{\ell,j}}

\newcommand{\tH}{\tilde{H}}

\newcommand{\B}{\mathcal{B}}

\newcommand{\ra}{\mathfrak{r}}

\newcommand{\tmu}{\tilde{\mu}}
\newcommand{\tnu}{\tilde{\nu}}
\newcommand{\F}{\mathring{F}}
\newcommand{\f}{\mathring{f}}

\newcommand{\vf}{\varphi}

\newcommand{\tvf}{\tilde{\varphi}}

\newcommand{\crs}{\mathcal{C}^r_s}
\newcommand{\cps}{\mathcal{C}^p_s}
\newcommand{\cqs}{\mathcal{C}^q_s}
\newcommand{\Lip}{\mbox{Lip}}
\newcommand{\loc}{\mbox{\scriptsize loc}}

\begin{document}


\title{Behavior of the Escape Rate Function in Hyperbolic Dynamical Systems}

\author{Mark Demers\thanks{Department of Mathematics and Computer Science,
Fairfield University, Fairfield, USA.  Email: mdemers@fairfield.edu.  This research
is partially supported by NSF grant DMS-0801139.} \and
\and Paul Wright\thanks{Department of Mathematics, University of Maryland, College Park, USA.  Email: paulrite@math.umd.edu.}
}

\maketitle

\begin{abstract}

For a fixed initial reference measure, we study the dependence of the escape rate on the hole for  a smooth or piecewise smooth hyperbolic map.
First, we prove
the existence and H\"older continuity of the escape rate for systems with small holes admitting
Young towers.  Then we consider general holes for Anosov diffeomorphisms, without size or Markovian restrictions.  We prove bounds on the upper and lower escape rates using the
notion of pressure on the survivor set and show that a variational principle holds under generic
conditions.  However, we also show that the escape rate function forms a devil's staircase
with jumps along sequences of regular holes and
present examples to elucidate some of the difficulties involved in formulating a general theory.

\end{abstract}

Consider a map $f:M \circlearrowleft$ of a measure space $M$ in which
a set $H \subset M$ is identified as a {\em hole}.  We keep track of a point's orbit until
it enters $H$; once this happens, it disappears and is not allowed to return.

One of the first quantities of interest for such open systems is the rate of escape of mass from the system
after it is initially distributed  according to a fixed reference measure, such as Lebesgue measure.  More precisely, given a
measure $\mu$ on $M$, the (exponential)   {\em  escape rate} of $\mu$ is $-\rho$,   where
\begin{equation}
\label{eq:escape def}
\rho=\rho(H,\mu) = \lim_{n \to \infty} \frac{1}{n} \log \mu(\cap_{i = 0}^n f^{-i} (M\backslash H)),
\end{equation}
if this limit exists.  We write  $\underline{\rho}$ and $\overline{\rho}$ for the $\liminf$ and $\limsup$ of the right hand side of \eqref{eq:escape def}, respectively,  so that $-\underline{\rho}$ is the upper  escape rate, and
$-\overline{\rho}$ is the lower escape rate.

Note that $\rho\leq 0$, so that the escape rate is non-negative.  Also, although we usually suppress the parameters, $\rho$ depends on both the hole, $H$, and on the measure,
$\mu$.  In the present work, $\mu$ will always be either Lebesgue measure or the Sinai, Ruelle, Bowen (SRB) measure for $f$, and we study the behavior of the escape rate as a function of the hole.

Many works on systems with holes have focused on the existence of quasi-invariant measures with
physical properties by assuming either the existence of
a finite Markov partition \cite{pianigiani yorke, cencova1, collet ms1, chernov mark1} or
that the holes are quite small \cite{chernov mt1, liverani maume, chernov bedem, bdm, dwy}.
The derivative of the escape rate in the
zero-hole limit has also received attention recently  \cite{buni yur, kl escape, ferg polli}.

In this paper we consider \emph{both} small and large holes and make no Markovian
assumptions on the dynamics of the open systems.
Let $H_t$ be a 1-parameter family of holes varying continuously
with $t$. Let $\rho(t)=\rho(H_t,\mu)$, when defined.
In this context, we focus on two related questions.

\begin{enumerate}
  \item[1.]{\bf Is $t \mapsto \rho(t)$ continuous?}  If so, does it possess  a higher degree of regularity?\footnote{If the escape rate is not defined on an open interval of $t$, we can ask about the continuity of $\underline{\rho}$ and $\overline{\rho}$.}
\end{enumerate}


\begin{enumerate}
  \item[2.]  {\bf
What is the overall structure of the escape rate function?}  Are there qualitative
differences in the escape rate function as we move out of the small hole
regime?
\end{enumerate}

The current understanding of the large hole regime is poor compared to the understanding of the small hole regime.  This is because, for small holes, we may consider the open system as a perturbation of the closed system, where no mass escapes.  In particular, perturbative spectral arguments have proven useful in studying this regime for certain systems, see for example \cite{liverani maume, demers liverani, keller liverani, dwy}.   In this paper, we extend such developments to
prove the existence and H\"older continuity of the escape rate for systems with small holes admitting
Young towers.  Young towers are often used to study hyperbolic systems where a clear spectral picture is unavailable for the original map but is available for a tower extension.  In order to prove our results,  we construct a perturbative framework in the tower setting and then show that the results proved for the tower map can be pushed back to the original system.
For this, we use the norms for the hyperbolic transfer operator on the tower introduced in
\cite{demers norms}.

A second motivation of the present work is to make further headway in the study of the large hole regime, without making overly restrictive assumptions on the hole, such as that it be an element of a finite Markov partition for $f$.  In order to do this, we  obtain bounds on the escape rate using a variational principle.  Such an approach for studying escape rates in uniformly hyperbolic systems was first introduced by Bowen \cite{bowen book} and
Young \cite{young escape}
and was recently developed for systems lacking uniform hyperbolicity \cite{dwy2}.

Since the general situation for systems with larger holes is unclear,  we  restrict our attention here to Anosov diffeomorphisms.  We prove general bounds on the upper and lower escape rates using the
notion of pressure on the survivor set and show that a variational principle holds under generic
conditions; however, we also show that the escape rate function can form a devil's staircase,
with jump discontinuities, for sequences of regular holes and
present examples to elucidate some of the difficulties involved in formulating a general theory.

\subsubsection*{Acknowledgments}

The authors would like to thank Lai-Sang Young for both starting and
encouraging them on this project.


\section{Statement of Results}
\label{results}

All of the results in this work concern a map $f:M\circlearrowleft,$  a $C^2$ or piecewise
$C^2$ map of a
Riemannian manifold.  Throughout, $\mu$ denotes Lebesgue measure on $M$,
which we assume is finite\footnote{When $f$ is only piecewise $C^2$, it is sometimes convenient to consider $M$ to be a countable union of disjoint components, and so our assumptions do not exclude this possibility.}.

We consider holes $H\subset M$ that are open.
The set of all such holes  in $M$ is denoted by $\ho$.
With $H \in \ho$ fixed, we define $M^n =  \cap_{i = 0}^n f^{-i} (M\backslash H)$ and $\f^n = f^n|_{M^n}$, for $n \geq 1$, to represent the dynamics of $f$
restricted to the set of points that have not escaped by time $n$.
Note that once $H$ is introduced, its boundary $\partial H$ must give rise to singularities for $\f$, even if $f$ is smooth.
To this end, if $\Si$ denotes the singularity
set for $f$ (which could be empty), we consider the enlarged singularity set
given by $\Si_H = \Si \cup \partial H $.

\bigskip
\noindent {\bf A. Systems Admitting Young Towers}

\medskip

Our first result addresses the regularity of the escape rate for small holes
when the system admits a Young tower as introduced in
\cite{young tower}.  We define these towers in detail in Section~\ref{tower}, but we will introduce them now, as certain technical aspects of these towers are necessary for precisely stating our result.

In brief, the tower map $F:\Delta\circlearrowleft$ is a type of countable Markov extension of $f:M\circlearrowleft$
built from a return time function $R$ defined on a hyperbolic set $\Lambda\subset M$.  Returns are only defined when the returning set has a hyperbolic Markov structure.

The decay rate of the quantity
$\mu(x \in \Lambda : R(x)>n)$  is
crucial to determining
the statistical properties of $f$.  All of the  towers we consider have {\em exponential tails}, i.e.~there exist constants $C>0$, $\theta <1$
such that $\mu(R>n) \leq C\theta^n$ for all $n \geq 0$.  In this case,
under certain mixing assumptions on $f$, it may be possible to construct the tower  so that a spectral gap exists for the transfer operator $\Lp_F$ associated with
$F$ acting on a certain
space of distributions.  By this, we mean that   1 is a simple eigenvalue for $\Lp_F$, and the rest of the spectrum is contained in a disk of radius $r<1$.  In this situation, an SRB measure $\nusrb$ for $f$ can be constructed on $M$ and  many strong statistical properties for $(f,\nusrb )$ can be derived (see \cite{young tower, demers norms}).

Let
$\pi: \Delta \to M$ denote the canonical projection satisfying
$f \circ \pi = \pi \circ F$.
We say a tower $(F, \Delta)$ {\em respects the hole} $H$ if the following conditions
are satisfied:
\begin{enumerate}  \vspace{-6 pt}
	\item[{\bf (H.1)}]  $\pi^{-1} H$ is the union of countably many elements in
	the Markov partition for $F$.  \vspace{-6 pt}
	\item[{\bf (H.2)}]   The set  $\Lambda$ from which the base of the tower  is constructed consists of points
	$x \in M\setminus H$ that approach $\Si_H$  slower than a fixed exponential rate, i.e.~there are constants $\delta > 0$, $\xi_1 > 1$, such that for all $n\geq 0$ and all $x\in\Lambda$,  $d(f^nx, \Si_H) \geq \delta \xi_1^{-n}$.   \vspace{-6 pt}
\end{enumerate}	

\noindent The notion of a tower respecting a hole has been used in a variety of settings
starting with \cite{demers tower hole}; for hyperbolic systems, the condition
{\bf (H.2)} was introduced in \cite{dwy}.  In applications, {\bf (H.1)} is ensured in part by adjoining the boundary of the hole to the singularity set, resulting in $\Si_H$.    Because this enlarged singularity set can cause unbounded changes in the return times to $\Lambda$,  towers that respect holes must be constructed separately, even when towers have previously been constructed for the system without consideration of the hole.

Once the hole has
been introduced, we can ask if the spectral gap of $\Lp_F$ persists for the transfer operator
$\Lp_{\F}$ corresponding to the tower map with hole, $\F$.
By the persistence of the  spectral gap for $\Lp_{\F}$, we mean
that $\Lp_{\F}$ has a simple real eigenvalue $\ra \leq 1$, and the rest of the spectrum
of $\Lp_{\F}$ is still contained in a disk of radius $r$, with $r < \ra$.

The existence of such a spectral gap implies that the escape rate  equals $-\log \ra$
for a large class of initial measures on $M$.  In addition, there is a unique
physical quasi-invariant probability measure $\mu_H$ that is the analog of the physical SRB measure for the system $f$ without holes.   This quasi-invariant measure satisfies $\f_* \mu_H = \ra \mu_H$, and there is a large class of initial measures on $M$ that limit on $\mu_H$ if they are pushed forward under the action of $\f$ and then renormalized back into probability measures (see \cite{demers young}
for a survey of  these techniques).

We call $H_0$ an {\em infinitesimal hole} if $\mu(H_0) = 0$ and $H_0$
can be realized as a limit of holes $H \in \ho$ in the topology
induced by the Hausdorff metric.

Let $\{ H_t \}_{t \in I} \subset \ho$, where $I$ is some interval, be a family
of holes. We call the family
{\em well-parametrized} if
\begin{enumerate}  \vspace{-4 pt}
	\item[(i)]   $0 \in I$ and $H_0$ is an infinitesimal hole;  \vspace{-6 pt}
	\item[(ii)]  dist$(\partial H_t, \partial H_{t'}) \leq |t - t'|$,
	where distance is measured in the Hausdorff metric.\vspace{-4 pt}
\end{enumerate}

Suppose that $f$ has a unique invariant SRB measure $\nusrb$.
In Section~\ref{proof of thm continuity} we prove the following theorem.

\begin{theorem} {\bf (Continuity of the Escape Rate for Small Holes)}
\label{thm:continuity}
 Let $\{ H_t \}_{t \in I}$ be
a well-parametrized sequence of holes in $\ho$.
Suppose that for each $t\in I$, $(f,M)$ admits a tower respecting $H_t$.  If for some
$\ve >0$, the towers satisfy the uniformity conditions {\bf(U)} of
Section~\ref{proof of thm continuity} for all $t \in [0,\ve]$, then
there exists $0 < \delta \le \ve$ such that
$\rho(t)=\rho (H_t, \nusrb)$ exists for each $t \in [0,\delta],$ and
$t \mapsto \rho(t) $ is a H\"older continuous function.
In addition, for $t \in [0, \delta]$ a unique physical quasi-invariant measure $\mu_{H_t}$ exists,
and these measures vary H\"older continuously
with $t$.
\end{theorem}

\noindent
The measures varying H\"older continuously means that there exist $C, \alpha >0$ such that,
\[
|\mu_{H_t}(\vf) - \mu_{H_{t'}}(\vf) | \le C |t-t'|^\alpha |\vf|_{C^0(M)} \; \; \; \mbox{ for all } \vf \in \C^0(M) .
\]

The basic idea of our proof is as follows:  Given $H_t$ and $H_{t'}$ with $|t-t'|$ small, we construct a tower over $f:M\circlearrowleft$ that respects {\em both} $H_t$ and $ H_{t'}$. On the tower, the lifted holes $\pi^{-1}H_t$ and $\pi^{-1}H_{t'}$ are close together, and so the spectral picture changes little when we consider one open transfer operator versus the other.   To
make this precise, we use the perturbative framework of \cite{keller liverani} and the
hyperbolic norms introduced in \cite{demers norms}.
Our use of these hyperbolic norms on $\Delta$
simplifies the argument greatly and allows us to avoid working with a second induced
object, the quotient tower $\bDelta$ (see Remark~\ref{rem:quotient}).

The analogue of Theorem~\ref{thm:continuity} has been shown to hold in \cite{demers liverani} for
Anosov diffeomorphisms and some piecewise hyperbolic systems with bounded derivative in two dimensions
using a different approach.

\medskip
\noindent
{\bf An application to the 2D periodic Lorentz gas.}
Let $f:M \circlearrowleft$ be the billiard map associated with a two-dimensional periodic
Lorentz gas with finite
horizon whose scatterers are bounded by $\C^3$ curves with strictly positive curvature.
In \cite{dwy}, holes are introduced into $M$ that are derived from two types of holes
in the billiard table $X$:  An open segment on the boundary of one of the scatterers or an open convex set in $X$ whose closure is disjoint from
any of the scatterers.  If $\sigma \subset X$ is one such hole, it induces a hole $H_\sigma \subset M$
which is labeled as a hole of Type I or Type II respectively.  For a detailed description
of the geometry of these holes in $M$, see \cite[Section 3.1]{dwy}.
The Young towers for this class of maps constructed in \cite{dwy} satisfy the assumptions of
Theorem~\ref{thm:continuity}, yielding the following corollary.

\begin{corollary}
Suppose $f$ is the billiard map described above and let $\{ H_t \}_{t \in I}$ be a
well-parametrized sequence of holes of Type I or Type II.  For $t$ sufficiently small,
both the escape rate $\rho(t)$ and the physical quasi-invariant measures $\mu_{H_t}$
vary H\"older continuously with $t$.
\end{corollary}


\vskip .2in
\noindent {\bf B. Large and Small Holes for Anosov Diffeomorphisms}

\medskip
Our next result concerns the general behavior of the escape rate as a function of both small and large holes.  Because the general picture for large holes is unclear, we limit our attention to Anosov diffeomorphisms.
Since we cannot rely on perturbative spectral arguments, we
use a variational principle to obtain
bounds on the escape rate.   Such an approach for investigating escape rates was first used in \cite{bowen book}; see \cite{dwy2} for a summary of further developments in this direction.

Let $M$ be a compact Riemannian manifold and let $f:M \circlearrowleft$ be a $C^{1+\epsilon}$  Anosov diffeomorphism.
Note that, unlike in Section {\bf A} above, for any $H \in \ho$,
$\Si_H = \partial H$ is now automatically compact, since $ M$ is.
As before, $\mu$ denotes Lebesgue measure on $M$.
The {\em survivor set}, $\Omega = \Omega(H) = \cap_{n \in \mathbb{Z}} f^n(M \setminus H)$, is the set of points whose forward and backward orbits never enter $H$.  We do not place any {\em a priori} assumptions on the mixing properties of $f$.

Given an $f$-invariant measure on $M$, we define
\[
P_\nu = h_\nu(f) - \int \chi^+ \, d\nu ,
\]
where $h_\nu(f)$ is the metric entropy of $\nu$, and $\chi^+$ is the
sum of positive Lyapunov exponents of $f$, counted with multiplicity.  In \cite{chernov mark1, chernov mark2}, the authors prove
that for holes that are elements of a finite Markov partition, an
{\em escape rate formula} holds, i.e.~there exists an ergodic, $f$-invariant measure $\nu$
such that $\rho(\mu) = P_\nu$.  These results were extended to small, non-Markov holes
by approximation in \cite{chernov mt2}.

Note that any $\f$-invariant measure must have its support contained inside $\Omega$, and so we define $\I(\Omega)$ to be the set of ergodic, $f$-invariant measures whose support is contained in $\Omega$.

Given a class of $f$-invariant measures $\mathcal{C}$ on $\Omega$, we define the
{\em pressure} on $\Omega$ to be
$\pa_{\mathcal{C}} = \sup_{\nu \in \mathcal{C}} P_\nu$.  We say that $f$ satisfies a {\em full variational principle} if
$\rho(\mu) = \pa_{\mathcal{C}}$, for an appropriate class of measures $\mathcal{C}$.  Ideally, the class should be as large as possible. Let $N_\ve(A)$ denote the $\ve$-neighborhood
of a set $A$.
Following \cite[Section 2.I]{dwy2}, we define
\[
\begin{split}
\G(\Omega) = \{ & \nu \in \I(\Omega) : \mbox{The following holds for $\nu$-a.e.\ $x$: given $\gamma > 0$, $\exists r = r(x,\gamma) > 0$ } \\
& \mbox{ such that $N_{re^{-\gamma i}}(f^ix) \subset M \setminus H$ for all $i \ge 0$} \} .
\end{split}
\]
A standard Borel-Cantelli argument shows that if
$\nu \in \I(\Omega)$ has the property that for some $C, \alpha >0$,
$\nu(N_\ve(\partial H)) \leq C\ve^\alpha$ for all
$\ve>0$, then $\nu \in \G(\Omega)$.

A full variational principle  and an escape rate formula are proved for hyperbolic systems
admitting Young towers respecting small holes in \cite{dwy2}, with $\G(\Omega)$  as the class of measures over which the supremum is taken.
In this section, we take a more general approach and study the relationship between
escape
rate and pressure without placing restrictions on the size or placement of the hole.
All results stated in this section are proved in Section~\ref{anosov}.

\begin{theorem}
\label{thm:anosov}
Let $f : M \circlearrowleft$ be a $C^{1+\ve}$ Anosov diffeomorphism with hole $H \in \ho$.  Then
$$
\pa_{\G}  \ \le \ \underline \rho (H,\mu)\ \le \ \overline \rho (H,\mu)\  \le \ \pa_{\I} \ .
$$
If in addition $\partial H \cap \Omega (H)=\emptyset$, then $\rho (H,\mu)$ is well-defined and equals $\pa_{\G}=\pa_{\I}$.
\end{theorem}

Note that since $\Omega(H)$ and $\partial H$ are both compact,
the condition $\partial H \cap \Omega (H)=\emptyset$ is equivalent
to $d(\partial H, \Omega(H)) > 0$, where $d$ is distance in the Hausdorff metric.  Hence, in this case, $\G=\I$.

In what follows, we consider holes in $\ho$ that have a finite number of connected components.   Let $ \hofin$ denote this collection; $\hofin$ inherits the topology induced by the Hausdorff metric.
In the two-dimensional case, we call a hole $H \in \hofin$ {\em regular} if its boundary is
comprised of a finite number of local stable and unstable manifolds.

Theorem~\ref{thm:anosov} points to the existence  of invariant measures
for $\f$ that give too much weight
to small neighborhoods of $\partial H$ as a source of potential problems.
Our next set of results illustrates two points.  (1) The
condition $\partial H\cap \Omega(H)= \emptyset$
is quite general:  It
holds for an open (and in dimension two, dense) set of holes and
for a full measure set of parameters along sequences of regular holes;
(2) exceptional situations do occur.  Indeed, these exceptions
cause the escape rate to vary - otherwise, it remains locally constant.
Proposition~\ref{prop:path} presents situations where the inequalities in
Theorem~\ref{thm:anosov} are strict.  Although the examples
are contrived, they need to be taken
into account in the formulation of general results.

\begin{proposition}
\label{prop:open dense}
Let $f:M \circlearrowleft$ be as in Theorem~\ref{thm:anosov}.  Then
\begin{enumerate}       \vspace{-6 pt}
  \item[(a)]  the set of $H \in \hofin$ that satisfy $\partial H\cap  \Omega(H) =\emptyset$ is
  open in $\hofin$;            \vspace{-6 pt}
  \item[(b)]  if dim$(M)=2$ and $f$ is topologically transitive, then the set of $H \in \hofin$ that satisfy
  $\partial H\cap \Omega(H)=\emptyset$ is dense in $\hofin$.
\end{enumerate}
\end{proposition}

\begin{remark}
In the setting of Proposition~\ref{prop:open dense}(b), for each
  $H \in \hofin$ and $\ve>0$,  in fact there exists a regular hole $H_\ve$  such that
  $d(H, H_\ve) < \ve$ and $d(\partial H_\ve, \Omega(H_\ve))>0$.
  This is shown in our proof.
\end{remark}

\begin{proposition}
\label{prop:sequence}
Let dim$(M)=2$, $f$ be topologically transitive and suppose $\{ H_t \}_{t\in I} \subset \hofin$ is a
sequence of regular holes with a fixed number of
connected components satisfying:
\begin{enumerate}
  \item[(i)]  the number of smooth components of $\partial H_t$ is uniformly bounded on $I$;
  \item[(ii)]  $t \mapsto \partial H_t$ is continuous;
  \item[(iii)]  For any subinterval $J \subseteq I$ and any curve $\gamma$ locally
  transverse to $\{ \partial H_t \}_{t \in J}$, if $E \subset J$ has positive Lebesgue measure
  in I, then $\{ \gamma \cap \partial H_t \}_{t \in E}$ has positive Lebesgue measure on $\gamma$.
\end{enumerate}
Then $\Omega(H_t) \cap \partial H_t = \emptyset$ for an open and dense set of $t \in I$ and
the exceptional set has zero Lebesgue measure in $I$.
\end{proposition}

\noindent Note that the sequence $\{ H_t \}_{t \in I}$ is neither assumed to converge to a point nor to be
monotonic.\begin{corollary}
\label{cor:devil}
Let $f: M \circlearrowleft$ and $\{ H_t \}_{t \in I}$ be as in Proposition~\ref{prop:sequence} and
let $-\underline \rho(t)$ and $-\overline \rho(t)$ denote the
upper and lower escape rates from $M \setminus H_t$ with respect to Lebesgue.
Then
\begin{enumerate}
  \item[(a)] $\rho(t)$ exists and is locally constant on an open and
full measure set of $t$.
\end{enumerate}
Now assume that $\{ H_t \}_{t \in I}$ is monotonically increasing.  Then
\begin{enumerate}
  \item[(b)] the functions $t \mapsto \underline \rho(t)$ and $\overline \rho(t)$
are monotonically decreasing  and each forms a devil's staircase, possibly with jumps;
  \item[(c)] $\underline \rho(\cdot )$ and $\overline \rho(\cdot )$
   are in general neither upper nor
lower semi-continuous once they are out of the small hole regime;
 \item[(d)]  if $\overline \rho(\cdot )$ is lower semi-continuous at $t$, then $\rho(t)$ exists;
 \item[(e)] if $\underline \rho(\cdot )$ is upper semi-continuous at $t$, then $\rho(t)$ exists.
\end{enumerate}
\end{corollary}
Taken together, statements (d) and(e) above imply that $\rho(t)$ typically exists even when
$\partial H_t \cap \Omega(H_t) \neq \emptyset$.  The only values of $t$ at which $\rho(t)$
may not exist are those at which $\overline \rho(t)$ and $\underline \rho(t)$ jump and
fail to be lower and upper semi-continuous, respectively.  This can occur at most
countably many times along the sequence.

\begin{remark}
If one considers the recent results in \cite{buni yur, kl escape, ferg polli}
regarding the existence of the derivative of $\rho(t)$ in the zero hole limit in a number of hyperbolic
settings, the picture of $\rho(t)$ that emerges from Corollary~\ref{cor:devil} is
rather surprising.  It indicates that along sequences of regular holes, $\rho(t)$ cannot be smooth on any interval containing
$0$: Indeed $\rho(t)$ cannot even be absolutely continuous on any interval  on which
it is not constant.
\end{remark}

\begin{proposition}
\label{prop:path}
There are examples of Anosov diffeomorphisms  with regular holes where
\[
\mbox{(a) $\pa_{\G} < \rho(\mu) = \pa_{\I}$; \qquad
(b) $\pa_{\G} = \rho(\mu) < \pa_{\I}$;  \; \; and \;
(c) $\pa_{\G} < \underline \rho(\mu) \leq \overline \rho(\mu) < \pa_{\I}$.}
\]
\end{proposition}

\begin{remark}
We conclude with an observation on large versus small holes in hyperbolic systems.
Evidently, no invariant measure
with pressure close enough to $0$ can be too concentrated near $\partial H$
for a large class of small holes since $t \mapsto \rho(t)$ is H\"older continuous in the setting of
Theorem~\ref{thm:continuity}.
On the other hand,
for larger holes invariant measures which maximize pressure
can live on $\partial H$ and create jumps in the escape rate as demonstrated
by  Corollary~\ref{cor:devil} and Proposition~\ref{prop:path}.
\end{remark}


\section{Proofs of Anosov Results}
\label{anosov}


\subsection{Proof of Theorem~\ref{thm:anosov}}

Recall that the conditions $\partial H \cap \Omega(H) = \emptyset$ and
$d(\partial H, \Omega(H)) > 0$ are equivalent since both sets are compact.
In this case $\G=\I$ so that $\pa_{\G} = \rho(H, \mu) = \pa_{\I}$
once the inequalities in the statement of Theorem~\ref{thm:continuity} are proved.
That $\pa_\G\leq \underline\rho(H,\mu)$ follows from \cite[Theorem A]{dwy2}.

To finish the proof, it suffices to show that $\overline\rho(H,\mu)\leq\pa_\I$.    To this end,  approximate $H$ by a sequence of increasing holes $H_n \subset H$, $H_n \in \ho$, such that each
$H_n$ is a union of finitely many elements of a Markov partition for $f$.
Since $f$ admits finite Markov partitions with arbitrarily small diameter, we
can choose the sequence so that $\cup_{n=1}^\infty H_n=H$.

We set $\mathring M_n = M \setminus H_n$ and in general denote by the subscript $n$ objects associated with $H_n$.  Note that $\mathring M_n$ is a decreasing sequence of
closed sets converging to $\mathring M = M \setminus H$.
The same is true of $\Omega_n$ and $\Omega$.  Since $H_n$ is a Markov hole, the escape rate $\rho_n=\rho(H_n,\mu)$
is well-defined and there exists $\nu_n \in \I(\Omega_n)$ such that
$\rho_n=P_{\nu_n}$ \cite{chernov mark2}.

Let $\nu$ be a limit point of the $\nu_n$.   Then $\nu$ is an $f$-invariant measure whose support is contained in $\Omega$, but $\nu$ need not be ergodic.  (Even if it were, it would not necessarally be an element of $\G$.)
Since $x\mapsto \log | \textrm{det}(Df\vert_{E^u})(x)|$ is a continuous function, we have
$\lim_{n \to \infty} \int \chi^+ d\nu_n = \int \chi^+ d\nu$.

In addition, $h_\nu(f) \ge \limsup_{n \to \infty} h_{\nu_n}(f)$ due to the expansiveness
of $f$.  We include the brief proof here for convenience.
Since $f$ is expansive, there exists $\ve > 0$ such that if $\xi$ is a finite
measurable partition of $M$ with diam$(\xi) < \ve$, then
$h_\eta(f, \xi) = h_\eta(f)$ for any invariant Borel measure $\eta$
\cite{bowen entropy}.  Fix such a partition $\xi$ with $\nu(\partial \xi) = 0$.
Let $H_\eta(\xi_k)$ denote the entropy of the partition
$\bigvee_{i=-k}^k f^i\xi$ with respect to a measure $\eta$, and for $\delta >0$ choose $k$
such that $\frac 1k H_\nu(\xi_k) \leq h_\nu(f) + \delta$.  Then
since $\frac 1k H_{\eta}(\xi_k)$ is a decreasing function of $k$ for any $\eta$, we have
\[
\limsup_{n \to \infty} h_{\nu_n}(f) = \limsup_{n \to \infty} h_{\nu_n}(f, \xi)
\leq \limsup_{n \to \infty} \frac 1k H_{\nu_n}(\xi_k)
=  \lim_{n \to \infty} \frac 1k H_\nu(\xi_k)
\leq h_\nu(f) + \delta,
\]
which proves the claim, since $\delta>0$ is arbitrary.

We have shown that
\[
P_\nu \ge \limsup_{n \to \infty} P_{\nu_n} = \limsup_{n \to \infty} \rho_n\ge \overline \rho,
\]
where the last inequality is true by monotonicity:
$\mathring M_n \supset \mathring M$ for each $n$.  By the ergodic decomposition,
there exists  a measure $\pi_\nu$ on $\I(\Omega)$ such that
$\nu = \int_{\I} \eta \, d\pi_\nu(\eta)$.  In fact, since $f$ and $\log |\det (Df|_{E^u})|$ are continuous,
we have $h_\nu(f) - \int \chi^+ d\nu = \int_{\I} (h_{\eta}(f) - \int \chi^+ d\eta) \, d\pi_\nu(\eta)$
(see for example \cite[Theorem 8.4]{walters}), so
there must exist an ergodic measure $\eta \in \I (\Omega)$ such that
$P_{\eta} \ge \overline{\rho}$.    This finishes the proof of Theorem~\ref{thm:anosov}


\subsection{Proof of Proposition~\ref{prop:open dense}}

Throughout this section we assume that $H\in \hofin$ has only one connected component.  The case of multiple connected components is handled similarly, one
component at a time.

As a reminder, because $\Omega$ and $\partial H$ are compact, the condition that $\partial H\cap \Omega (H) =\emptyset$ is equivalent to $d(\partial H, \Omega(H))>0$.

To prove statement (a), choose $H \in \hofin$ such that $d(\partial H, \Omega(H))>0$.
Let $B(x,\varepsilon)$ denote the ball of radius $\varepsilon$ centered
at $x$. For each $x \in \partial H$, let $\varepsilon(x)$ and $n(x)$ be such
that $f^{n(x)}(B(x,\varepsilon(x))) \subset H$.  Since $\partial H$ is compact, we may
choose $x_1, \ldots, x_k$ so that
$U_0:=\cup_{i=1}^k B(x_i, \frac12 \varepsilon(x_i)) \supset \partial H$. Notice that
$V=\cup_{i=1}^k f^{(n(x_i))}(B(x_i, \frac12 \varepsilon(x_i))) \subset H$ is a positive distance
from $\partial H$.
Let $U_1 \subseteq U_0$ be a neighborhood of $\partial H$ that is a positive distance from $V\cup\Omega$.

Now consider a hole $H' \in \hofin$ with survivor set $\Omega'$ such that
$d(H,H') < \delta$ for some $\delta>0$.  By taking $\delta$ sufficiently small, we
can ensure that
$[(H'\setminus H)\cup(H\setminus H')]\subset U_1$.   If $x \not\in \Omega'$,
then $f^j x\in H'$ for some $j$.  Either $f^j x\in H$ or $f^j x\in U_0$, and in the latter case a further iterate of $x $ eventually lies in $V \subset H \cap H'$.  In both cases, $x\not\in \Omega$, so $\Omega \subset \Omega'$.   Similarly, $\Omega' \subset \Omega$.
To conclude, because $\partial H'  \subset U_1$ and $U_1$ is a positive distance from
$\Omega = \Omega'$, we  have $\partial H'\cap \Omega'=\emptyset$.

To prove
statement (b), fix $H \in \hofin$ and $\ve > 0$ sufficiently small that $H$ contains a
ball of diameter at least $5\ve$.
We approximate $\partial H$ by a union of finitely many
stable and unstable manifolds $\Gamma' = (\cup_i \gamma^s_i) \cup (\cup_i \gamma^u_i)$
such that $d(\Gamma', \partial H) < \ve/2$ where $\gamma^{s(u)}_i$ denotes the $i$th
(un)stable manifold.  Let $B_\ve \subset H$ be an open convex set such that
$d(B_\ve, H) > 2\ve$.

Now fix $\gamma^s_i$ and let $N_{\ve}(\gamma^s_i)$ denote the $\ve$ neighborhood of
$\gamma^s_i$ in $M$.  By transitivity of $f$, $\mu$-almost every $x \in N_{\ve/2}(\gamma_i^s)$
has a dense orbit.  Choose such an $x$ and let $\gamma^s_x$ be the local stable manifold through
$x$ of length at least $2|\gamma^s_i|$.  There exists an integer $n_x$ and an open set
$U_x \supset \gamma^s_x$ such that $f^{n_x}(U_x) \subset B_\ve$.

Now we modify $\Gamma'$ by replacing $\gamma^s_i$ with $\gamma^s_x$ and possibly
lengthening or shortening the $\gamma^u_j$ adjacent to $\gamma^s_i$ so that the ends
of $\gamma^u_j$ which formerly ended on $\gamma^s_i$ now end on $\gamma^s_x$.
We trim the ends of $\gamma^s_x$ as necessary.  We continue this process with each
$\gamma^s_i$ in forward time and each $\gamma^u_i$ in backward time.  In this way, we
construct $\Gamma$, a simple closed curve made up of finitely many stable and unstable manifolds $\gamma_k$,
which enjoys the following properties:  (i) $d(\Gamma, \partial H) < \ve$; (ii) for each
$\gamma_k$ there exists an open set $U_k \supset \gamma_k$
and $n_k \in \mathbb{Z}$ such that $f^{n_k}(U_k) \subset B_\ve$.  Thus letting
$H_{\Gamma}$ denote the hole with boundary $\Gamma$ and noting
that $B_\ve \subset H_\Gamma$ by construction, we have
$d(\partial H_{\Gamma}, \Omega(H_{\Gamma})>0$ as required.


\subsection{Proof of Proposition~\ref{prop:sequence}}
\label{sequence proof}

Let $\{ H_t \}_{t \in I}$ be a sequence of holes as described in the statement of
Proposition~\ref{prop:sequence}.  That the condition
$\Omega(H_t) \cap \partial H_t = \emptyset$ holds on an open set in $I$ follows from
Proposition~\ref{prop:open dense} since by assumption (ii) on the sequence,
$H_t$ varies continuously with $t$ in the Hausdorff metric.

Thus it suffices to show
that the exceptional set has Lebesgue measure 0 in $I$, for then an open set of full
Lebesgue measure is necessarily dense.
We do this by contradiction.  Let $E = \{ t \in I : \Omega(H_t) \cap \partial H_t \neq \emptyset \}$
and suppose $\ell(E)>0$, where $\ell$ denotes Lebesgue measure on $I$.

We partition $\{\partial H_t \}_{t \in I}$ into finitely many boxes $B_i$ sufficiently small
that all the elements of $\{ \partial H_t \cap B_i \}_{t \in I}$ are roughly parallel.
This is possible due to assumption (i) and the fact that the stable and unstable foliations are
H\"older continuous.  Let $E_i = \{ t \in I : \Omega(H_t) \cap \partial H_t \cap B_i \neq \emptyset \}$.
At least one of these sets must satisfy $\ell(E_i)>0$.  Fix one such index and call it $k$.
We suppose without loss of generality that the curves $\partial H_t$ in $B_k$
are all local stable manifolds.

Draw a curve $\gamma$ in $B_k$ that is uniformly transverse to each curve
$\partial H_t$ lying in $B_k$.
We choose one point
$x_t \in \Omega_t \cap \partial H_t \cap B_k$ for each $t \in E_k$.
Due to property (iii) of $\{ \partial H_t \}$,
we have
$\mu_\gamma(\gamma \cap W^s_{\loc}(x_t) : t \in E_k ) > 0$, where $\mu_\gamma$
denotes arclength on $\gamma$.

Note that if the forward orbit of $x$ is dense, then the entire stable manifold of
$x$ eventually falls into any open hole.  Since $x_t \in \Omega(H_t)$, the
curve $W^s_{\loc}(x_t)$ cannot contain any forward dense points, for then $x_t$ would fall into
$H_t$ under forward iteration as observed above.
Integrating these curves over $t \in E_k$, we see that
a positive $\mu$-measure set of points in $B_k$ do not have a dense forward orbit.
This contradicts the fact that points with dense forward orbits have full measure in $M$.
Thus $\ell(E)=0$ as required.

The case in which the curves $\partial H_t$ in $B_k$ are local unstable manifolds is
handled similarly using the fact that points with dense backward orbits have full
measure as well.


\subsection{Proof of Proposition~\ref{prop:path}}

(a) Let dim$(M)=2$ and let $p$ be a fixed point for $f$ with expanding
eigenvalue $\lambda >1$, i.e., $f$ is orientation preserving.
Let $W^s_{\loc}(p)$ denote the
local stable manifold through $p$.  Let $U$ be a neighborhood of $p$, divided into
two halves by $W^s_{\loc}(p)$: the left half is in $H$ and the right
one not, so that $W^s_{\loc}(p) \subset \partial H$. From this alone, we see that
$\underline \rho \ge - \log \lambda$.  We now construct $f$ and $H$ with two
additional features:  (i) Except for $W^s_{\loc}(p)$, all points in the right half
of $U$ eventually fall into the hole; this can be arranged by having $W^u(p)$
run into $H$. (ii) By taking $\lambda$ close enough to 1 and $H$
large enough, we can arrange for $\pa_{\I_p} < - \log \lambda$ where
$\I_p \subset \I$ is the set of ergodic invariant measures supported on
$\Omega \setminus \{ p \}$.
Thus $\pa_{\I} = - \log \lambda$ and as noted earlier
$\underline \rho \ge - \log \lambda$.  But
$\overline \rho \le \pa_{\I}$ by Theorem~\ref{thm:anosov} so that
$\rho$ is well-defined and $\rho = \pa_{\I} = - \log \lambda$.
Also, since $\G \subset \I_p$, we have $\pa_{\G} < \rho$.

\medskip \noindent
(b) We use the same setup as in (a), but now $\lambda<-1$, i.e.,
$f$ is orientation reversing.
Now $p \in \Omega$ as before, but both halves of $U$ on either side of
$W^s_{\loc}(p)$ fall into $H$ in finitely many steps so that
the escape rate is unrelated to the eigenvalue at $p$.  We
choose $\lambda$ close enough to $-1$ that $\overline \rho < \pa_{\I} = - \log |\lambda|$.
We further require that: (i) $H$ is a union of elements of a Markov partition for $f$;
(ii) $(\partial H \setminus W^s_{\loc}(p)) \cap \Omega = \emptyset$.  By (i) and the
results of \cite{chernov mark2}, $\rho(\mu)$ is well-defined and
there exists an ergodic invariant measure $\nu$ supported
on $\Omega$ such that $P_\nu = \rho$.  By (ii), $\nu \in \G$ so that by
Theorem~\ref{thm:anosov}, $\pa_{\G} = \rho$.

\medskip \noindent
(c) We combine the two behaviors described in parts (a) and (b).   Assume
$f$ is orientation preserving.  Let $p$ be a fixed point for $f$ and let
$q, q'$ be an orbit of period two.
Let $\mu_p$ denote the point mass at $p$ and let $\nu_q$ denote the
invariant measure supported on $\{ q, q' \}$.
Let $V_q$ denote a neighborhood of $q$ divided into two halves by
$W^s_{\loc}(q)$.  Orienting stable manifolds in an approximately vertical direction,
we label these two halves as $V_q^\ell$ and $V_q^r$ for left and right.  We define
analogous objects for $q'$.

We choose a regular hole $H$ with the
following properties:  (i)  $W^s_{\loc}(p) \subset \partial H$ and the neighborhood
$U$ of $p$ is as described in (a);
(ii) $W^s_{\loc}(q) \cup W^s_{\loc}(q') \subset \partial H$,
$V_q^r \subset H$ and $V_{q'}^\ell \subset H$.
Letting $\lambda_p$ denote the expanding eigenvalue at $p$ and using the same
reasoning as in (a), we see from (i) that $\underline \rho \geq - \log \lambda_p$.
We let $\I' = \I \setminus \{ \mu_p, \nu_q \}$ and choose $\lambda_p$ close enough
to 1 such that $\pa_{\I'} < - \log \lambda_p$ so that
$\pa_{\G} \leq \pa_{\I'} < \underline \rho$.

Since $f$ is orientation preserving, $f(V_q^\ell) \subset H$ and
$f(V_{q'}^r) \subset H$ so that the full measure of $V_q \cup V_{q'}$ has
escaped after one step.  Thus $\overline \rho$ is independent of
$\log \lambda_q$, the expanding Lyapunov exponent on the orbit $q, q'$.  Taking
$\lambda_q$ close enough to 1, we can force
$\overline \rho < - \log \lambda_q \leq \pa_{\I}$.


\subsection{Proof of Corollary~\ref{cor:devil}}

{\em Proof of (a).}
It follows from Theorem~\ref{thm:anosov} and
Proposition~\ref{prop:sequence}
that $\rho(t)$ is locally
constant on an open and full measure set of $t$ since $\rho(t)$ exists and is locally constant
around any $t$ satisfying $d(\partial H_t, \Omega(H_t)) >0$.

\medskip
\noindent
{\em Proof of (b).}
Monotonicity of $\{ H_t \}$ clearly implies monotonicity of both
$\underline \rho_t$ amd $\overline \rho_t$ and their characterization as a devil's staircase
follows since the derivative of each function exists and equals zero on an open subset
of $I$ of full measure by part (a).

\medskip
We assume for the remainder of the proof that $\{ H_t \}_{t \in I}$ forms an increasing sequence
of holes so that $\underline \rho(t)$ and $\overline \rho(t)$ are (nonstrictly)
decreasing functions of $t$.  We prove (d) and (e) first and leave statement (c) for last.

\smallskip
\noindent
{\em Proof of (d).}
Assume $\overline \rho$ is lower semi-continuous at $t_0$.  By part (b), there exists
a sequence $t_n \downarrow t_0$ such that $\overline \rho(t_n) \uparrow \overline \rho(t_0)$,
and by part (a), the $t_n$ can be chosen so that $\rho(t_n)$ exists for each $n$,
i.e. $\underline \rho(t_n) = \overline \rho(t_n)$.  Thus
\[
\overline \rho(t_0) = \lim_{n \to \infty} \overline \rho(t_n) = \lim_{n \to \infty} \underline \rho(t_n)
\leq \underline \rho(t_0)
\]
where in the last inequality we have used monotonicity of $\underline \rho$ since
$t_n > t_0$ implies $H_{t_n} \supset H_{t_0}$ for each $n$.
Since $\underline \rho(t_0) \leq \overline \rho(t_0)$
by definition, we have $\underline \rho(t_0) = \overline \rho(t_0)$ so that $\rho(t_0)$ exists.

\medskip
\noindent
{\em Proof of (e).}
This is similar to part (d) except that we choose a sequence
$t_k \uparrow t_0$ such that $\underline \rho(t_k) \downarrow \underline \rho(t_0)$
at a point where $\underline \rho$ is upper semi-continuous.

\medskip
\noindent
{\em Proof of (c).}
Finally, we show that $\underline \rho(t)$ and $\overline \rho(t)$ can in fact have jumps
at particular values of $t$.
We refer to the examples constructed in the proof of
Proposition~\ref{prop:path}.  All notation is as in that proof.

\smallskip \noindent
{\em Violation of lower semicontinuity.}
Let $H_{t_0}$ be a hole satisfying the requirements of case (a) in the proof of
Proposition~\ref{prop:path}.  For $t < t_0$, take
$H_t$ to have the same
boundary as $H_{t_0}$ except for $W^s_{\loc}(p)$.
Here, we make the boundary of
$H_t$ be a local stable manifold running parallel to $W^s_{\loc}(p)$ lying inside
$U \cap H_{t_0}$
and varying continuously with $t$.
Notice that since $\rho(t_0) = - \log \lambda$ in this example,
shrinking $H_t$ in this way
does not change the escape rate for $t$ close to $t_0$.  Also,
$\Omega(H_t) = \Omega(H_{t_0})$ since $U \cap \Omega(H_t) = \emptyset$.
Thus $\rho(t) = - \log \lambda$ for all
$t \in (t_0 - \ve,  t_0]$
for $\ve$ sufficiently small.

On the other hand, for $t > t_0$, we replace $W^s_{\loc}(p)$ by a local stable
manifold running parallel to $W^s_{\loc}(p)$ lying in $U \setminus H_{t_0}$ so that
$p$ is no longer in $\Omega(H_t)$.  Then
$\I(H_t) = \I_p(H_{t_0})$ so that for some $\delta >0$ and all $t > t_0$,
$\pa_{\I(H_t)}  < - \log \lambda - \delta = \rho(t_0) - \delta$
by construction of $p$.
By Theorem~\ref{thm:anosov}, we have
$\underline \rho(t) \leq \overline \rho(t) \leq \pa_{\I(H_t)}
< \rho(t_0) - \delta$ for all $t > t_0$.

\smallskip \noindent
{\em Violation of upper semicontinuity.}
Let $H_t$ be the same as described in the previous step, except that
$H_{t_0}$ is a hole satisfying the requirements of case (b) in the proof
of Proposition~\ref{prop:path}, i.e., the case when $f$ is orientation reversing.
Now the holes for $t > t_0$ satisfy $\rho(t) = \rho(t_0)$
since the full measure of $U$ disappears in one step for all the $H_t$.
For $t < t_0$, now a neighborhood
of $p$ survives for arbitrarily many steps so that $\lambda$ dominates the escape,
i.e., we choose $\lambda$ so that
$\overline \rho(t) \ge \underline \rho(t) \ge - \log |\lambda| > \rho(t_0) + \delta$ for
some $\delta > 0$.


\section{Proof of Theorem~\ref{thm:continuity}}
\label{tower}

We begin by reviewing some facts about Young towers from
\cite{young tower, demers norms}.  We then use these facts to prove
new results on the tower that we use to prove Theorem~\ref{thm:continuity}
in Section~\ref{proof of thm continuity}.


\subsection{Generalized horseshoe respecting $H$}
\label{horseshoe}

We recall the notion of a generalized horseshoe with infinitely many branches
and variable return times.  The existence of such a horseshoe leads immediately
to the definition of a Young tower.
We summarize
here only the most important properties and refer the reader to
\cite[Section 1.1]{young tower} for full details.

Following the notation in Section~1.1 of \cite{young tower}, we consider a smooth or
piecewise smooth map $f: M \to M$, and let $\mu$ and $\mu_\gamma$
denote respectively the Riemannian measure on $M$ and on $\gamma$
where $\gamma \subset M$ is a submanifold. We say the pair
$(\Lambda, R)$ defines a {\it generalized horseshoe} if {\bf
(P1)}--{\bf (P5)} below hold (see \cite{young tower} for precise formulation):

\begin{itemize}
\item[{\bf (P1)}] $\Lambda$ is a compact subset of $M$ with a hyperbolic product
structure, {\it i.e.}, $\Lambda = (\cup \Gamma^u) \cap (\cup
\Gamma^s)$ where $\Gamma^s$ and $\Gamma^u$ are continuous families
of local stable and unstable manifolds, and $\mu_{\gamma}\{\gamma
\cap \Lambda\}>0$ for every $\gamma \in \Gamma^u$.
\end{itemize}
A set $A$ is an $s$-subset (resp.\ $u$-subset) of $\Lambda$  if $\gamma \cap A \neq \emptyset$ implies $\gamma \subseteq A$ for any $\gamma \in \Gamma^{s(u)}$.
\begin{itemize}
\item[{\bf (P2)}] $R: \Lambda \to {\mathbb Z}^+$ is a {\it return time function}
to $\Lambda$. Modulo a set of $\mu$-measure zero, $\Lambda$ is the
disjoint union of $s$-subsets $\Lambda_j, j=1,2, \cdots,$ with the
property that for each $j$, $R|_{\Lambda_j}=R_j \in {\mathbb Z}^+$
and $f^{R_j}(\Lambda_j)$ is a $u$-subset of $\Lambda$.
Moreover, for each $n$, the number of
$j$ such that $R_j=n$ is finite.
\end{itemize}
We refer to elements of $\Gamma^{u(s)}$ by $\gamma^{u(s)}$ and
$|\det Df^u|$ denotes
the unstable Jacobian of $f$ with respect to $\mu_{\gamma^u}$.  Denote by
$\gamma^s(x)$ and $\gamma^u(x)$ the stable
and unstable leaves through $x$, respectively.  For $x,y \in \Lambda$, there exists a separation time $s_0(x,y)$, depending only on the unstable coordinate, and numbers $C_0 \geq 1$, $\alpha<1$ independent
of $x, y$, such that the following hold.
\begin{itemize}
\item[{\bf (P3)}] For $y \in \gamma^s(x)$, $d(f^nx,f^ny) \le C_0 \alpha^n d(x,y)$
    for all $n \geq 0$.\footnote{As an
    abstract requirement, (P3) is slightly stronger than the inequality
    $d(f^nx, f^ny) \leq C_0 \alpha^n$ stated in \cite{young tower}.  In practice, however,
    the stronger version holds for all systems for which Young towers have been
    constructed to date.}
\item[{\bf (P4)}] For $y \in \gamma^u(x)$ and $0 \le k \le n < s_0(x,y)$,

(a) $d(f^nx,f^ny) \le C_0 \alpha^{s_0(x,y)-n}$;

(b) $\log \Pi_{i=k}^n \frac{\det Df^u(f^ix)}{\det Df^u(f^iy)} \ \le
\ C_0 \alpha^{s_0(x,y)-n}.$

\item[{\bf (P5)}] (a) For $y \in \gamma^s(x)$,
$\log \Pi_{i=n}^\infty \frac{\det Df^u(f^ix)}{\det Df^u(f^iy)} \ \le
\ C_0 \alpha^n$ for all $n \ge 0.$

(b) For $\gamma, \gamma' \in \Gamma^u$, if $\Theta : \gamma \cap
\Lambda \to \gamma' \cap \Lambda$ is defined by
$\Theta(x)=\gamma^s(x) \cap \gamma'$, then $\Theta$ is absolutely
continuous and $\frac{d(\Theta_*^{-1}\mu_{\gamma'})}{d\mu_\gamma}(x)
\ = \ \Pi_{i=0}^\infty \frac{\det Df^u(f^ix)}{\det Df^u(f^i\Theta
x)}$.
\end{itemize}

The meanings of the last three conditions are as follows: Orbits
that have not ``separated" are related by local hyperbolic
estimates; they also have comparable derivatives. Specifically, {\bf
(P3)} and {\bf (P4)}(a) are (nonuniform) hyperbolic conditions on
orbits starting from $\Lambda$. {\bf (P4)}(b) and {\bf (P5)} treat
more refined properties such as distortion and absolute continuity
of $\Gamma^s$, conditions that are known to hold for
$C^{1+\varepsilon}$ hyperbolic systems.

We say the generalized horseshoe $(\Lambda, R)$ has {\it exponential
return times} if there exist $C>0$ and $\theta>0$ such that for
all $\gamma \in \Gamma^u$, $\mu_\gamma\{R>n\} \le C \theta^n$
for all $n \ge 0$.

\medskip
The setting described above is that of \cite{young tower}; it does not involve
holes. In this setting, we now identify a set $H \subset M$ (to be
regarded later as the hole) and introduce a few relevant
terminologies. Let $(\Lambda, R)$ be a generalized horseshoe for $f$
with $\Lambda \subset (M \setminus H)$.

Recall that $(\Lambda, R)$ {\it respects} $H$ if it satisfies conditions
{\bf (H.1)} and {\bf (H.2)} of
Section~\ref{results}.A.  In particular, {\bf (H.1)} says that for every $i$ and every
$\ell$ with $0 \le \ell \le R_i$, $f^\ell(\Lambda_i)$ either does
not intersect $H$ or is completely contained in $H$.

When constructing the horseshoe which respects the hole in the sense of {\bf (H.1)},
we still keep track of orbits that pass through $H$, i.e. we construct a horseshoe for the
closed dynamical system $(f,M)$ including $\partial H$ as part of the singularity set, but
we do not allow any escape at this stage.  We do this to facilitate comparison between the dynamics of the open system and those of the closed system.

We say the horseshoe $(\Lambda, R)$ is {\it mixing} if g.c.d.$\{R\} = 1$.
 If $(\Lambda, R)$ respects
$H$, let $R_H$ denote the restriction of $R$ to those $\Lambda_i$ which return to $\Lambda$
before entering $H$.  Then when we treat $H$ as a hole, we say the {\it surviving
dynamics} are {\it mixing} if in addition g.c.d.$\{ R_H \} = 1$.


\subsection{From generalized horseshoes to Young towers}

It is shown in \cite{young tower} that given a map $f: M \to M$ with a generalized horseshoe
$(\Lambda, R)$, one can associate a Markov extension
$F: \Delta \to \Delta$ which focuses on the return dynamics to $\Lambda$
(and suppresses details between returns) .
We first recall some facts about this very general construction, taking
the opportunity to introduce some notation.

Let
\[
\Delta = \{ (x,n) \in \Lambda \times \mathbb{N} :  n < R(x) \} ,
\]
and define $F: \Delta \to \Delta$ as follows: For
$\ell <R(x)-1$, we let $F(x,\ell) = (x, \ell+1)$, and define
$F(x,R(x)-1) = (f^{R(x)}(x),0)$. Equivalently, one can view
$\Delta$ as the disjoint union
$\cup_{\ell \ge 0} \Delta_\ell$ where
$\Delta_\ell$, the $\ell^{\mbox{\tiny th}}$ level of the tower, is a copy of
$\{x \in \Lambda: R(x)> \ell\}$. This is the representation we will use.
There is a natural projection $\pi: \Delta \to M$ such that $\pi \circ F = f \circ \pi$.
In general, $\pi$ is not one-to-one, but for each $\ell \ge 0$,
it maps $\Delta_\ell$ bijectively onto $f^\ell(\Lambda \cap \{R \ge \ell\})$
if $f$ is invertible.
We say $(F,\Delta)$ is mixing if $(\Lambda, R)$ is mixing.

In the construction of $(\Lambda,R)$, one usually introduces an increasing
sequence of partitions of $\Lambda$ into $s$-subsets representing
distinguishable itineraries in the first $n$ steps.
These partitions
induce a partition $\{ \Delta_{\ell, j} \}$ of $\Delta$ which is finite on
each level $\ell$ and is a (countable) Markov partition for $F$.
We define a separation time $s(x,y) \leq s_0(x,y)$
by $\inf \{ n>0: \mbox{$F^nx, F^ny$ lie in
different $\Delta_{\ell,j}$} \}$.

We borrow the following language from $(\Lambda, R)$ for use
on $\Delta$: For each $\ell,j$, recall that $\Gamma^s(\pi(\Delta_{\ell, j}))$ and
$\Gamma^u(\pi(\Delta_{\ell, j}))$ are the stable and unstable families
defining the hyperbolic product set $\pi(\Delta_{\ell, j})$. We will say
$\tilde \gamma \subset \Delta_{\ell, j}$ is an {\it unstable leaf} of
$\Delta_{\ell, j}$ if $\pi(\tilde \gamma) = \gamma \cap \pi(\Delta_{\ell, j})$
for some $\gamma \in \Gamma^u(\pi(\Delta_{\ell, j}))$, and use
$\Gamma^u(\Delta_{\ell, j})$ to denote the set of all such $\tilde \gamma$.
Let $\Gamma^u(\Delta) = \cup_{\ell,j}
\Gamma^u(\Delta_{\ell, j})$ be the set of all unstable leaves of $\Delta$.
{\it Stable leaves} of $\Delta_{\ell, j}$ and the families $\Gamma^s(\Delta_{\ell, j})$
and $\Gamma^s(\Delta)$ are defined
similarly.

The measures $\mu_\gamma$, $\gamma \in \Gamma^u(\Delta_0)$
are extended to $\Delta_\ell$, $\ell>0$, by defining
$\mu_\gamma(A) = \mu_\gamma(F^{-\ell}A)$ for all measurable $A \subset \Delta_\ell$.
Thus $J_{\mu_\gamma} F$, the Jacobian of $F$ with respect to $\mu_\gamma$,
satisfies $J_{\mu_\gamma}F \equiv 1$ except at return times.
Let $J^u\pi$ denote the Jacobian of $\pi$ with respect to the measures
$\mu_\gamma$ on $\Delta$ and $M$ respectively.  Then
$F$ enjoys properties {\bf (P3)}-{\bf (P5)} at return times due to the
identity $J^u\pi(F) J_{\mu_\gamma}F = |\det Df^u(\pi)| J^u\pi$ and the fact that $J^u\pi \equiv 1$
on $\Delta_0$.


\subsubsection{A reference measure on $\Delta$}
\label{good measure}

In each $\dlj$ we choose a representative leaf
$\hat{\gamma} \in \Gamma^u(\dlj)$.  For any $\gamma \in \Gamma^u(\dlj)$, let
$\Theta_{\gamma,\hat{\gamma}} : \gamma \to \hat{\gamma}$
denote the holonomy map along
$\Gamma^s$-leaves, i.e.\
$\Theta_{\gamma, \hat{\gamma}}(x) = \gamma^s(x) \cap \hat{\gamma}$.
It will be convenient to define a new reference measure along unstable leaves,
$m_\gamma$, by $dm_\gamma = \phi d\mu_\gamma$ where
$\phi(x)= \prod_{i=0}^{\infty} \frac{J_{\mu_\gamma}F(F^ix)}
                 {J_{\mu_\gamma}F(F^i(\Theta_{\gamma, \hat{\gamma}} x))}$.
Given $\gamma' \in \Gamma^u(\Delta_0)$, if $\gamma \in \Gamma^u(\dlj)$
satisfies $F(\gamma \cap S) = \gamma'$ for some $s$-subset $S$, then for
$x \in \gamma \cap S$, define
$J_\gamma F(x) = \frac{d(m_{\gamma'}\circ F)}{dm_\gamma}$.
Elsewhere on $\Delta$, $J_\gamma F \equiv 1$.
Similarly, one defines $J_\gamma F^R(x)$ whenever $F^R(\gamma \cap S)
= \gamma'$ for some $s$-subset $S$.
For convenience, we restate Lemma 1 from \cite{young tower}, which summarizes the
important properties of $m_\gamma$.

\begin{lemma}{\cite{young tower}}
\label{lem:jacobian}
Let $\gamma$, $\gamma' \in \Gamma^u(\dlj)$.
\begin{itemize}
  \item[(1)] Let $\Theta_{\gamma, \gamma'}: \gamma \rightarrow \gamma'$
be the holonomy map along $\Gamma^s$-leaves as above.
Then $\Theta_*m_\gamma = m_{\gamma'}$.
  \item[(2)] $J_\gamma F(x) = J_{\gamma'} F(y)$, $\forall x \in \gamma$, $ y \in \gamma^s(x) \cap \gamma'$.
  \item[(3)] $\exists C_1>0$ such that
    $\forall  x,y \in \gamma$ with $s_0(x,y) \geq R(x)$,
    $ \left| \frac{J_\gamma F^R(x)}{J_\gamma F^R(y)} - 1\right|
    \leq C_1 \alpha^{s(F^Rx, F^Ry)/2}$ . \vspace{-6 pt}
\end{itemize}
Moreover by (P5)(a),
$e^{-C_0} \leq \phi \leq e^{C_0}$ .
\end{lemma}

On $\Delta_0$,  we choose a transverse measure $m^s$ on $\Gamma^u(\Delta_0)$
normalized so that $m^s(\Gamma^u(\Delta_0))=1$.  Using the fact that
$F: \Delta_\ell \to \Delta_{\ell + 1}$ is simply rigid translation, we extend $m^s$
to each $\Gamma^u(\dlj)$.
We define $m$ to be the measure with factor measure $m^s$ and
measures $m_\gamma$ on unstable leaves.
Notice that in any $\dlj$, Lemma~\ref{lem:jacobian}(1) implies that
$m_\gamma(S) = m(S)$ for any s-subset $S \subseteq \dlj$ and
$\gamma \in \Gamma^u(\dlj)$.  This feature of $m_\gamma$
implies that $m$ is a product measure on each $\dlj$.
When disintegrating $m$ on a particular $\dlj$, we maintain the convention that
$m^s$ is normalized, but $m_\gamma$ is not.

\begin{remark}
\label{rem:quotient}
At this point in the application of Young towers, it is usual to define the quotient
tower $\bDelta = \Delta/ \!\! \sim$ where $x \sim y$ if $y \in \gamma^s(x)$.  The resulting
quotient system $(\barF, \bDelta)$ is expanding, allowing one to bring to bear the
usual analysis of the transfer operator for expanding systems.  However, since this
requires the extra step of lifting our results from $\bDelta$ to $\Delta$ before projecting
down to $M$, we find it simpler to work with the hyperbolic transfer operator
on $\Delta$ directly, which we do below.  Our approach also yields stronger results regarding
the H\"older continuity of the quasi-invariant measures.
\end{remark}


\subsubsection{Transfer operator}

We define a metric along stable leaves which makes the distance between unstable
leaves uniform.  Fix $x \in \Delta_0$ and let
$y \in \gamma^s(x)$.   Let $\Theta: \gamma^u(x) \to \gamma^u(y)$ be the sliding
map along stable leaves as above.  Define
$d_s(x,y) := \sup_{z \in \gamma^u(x)} d(z, \Theta z) $.
We extend this metric to $\Delta_\ell$, $\ell>1$, by setting
$d_s(F^\ell x,F^\ell y) = \alpha^\ell d_s(x,y)$
for all $\ell < R(x)$ and $y \in \gamma^s(x)$.
By {\bf (P3)},
\begin{equation}
\label{eq:ds contract}
d_s(F^nx, F^ny) \leq C_0 \alpha^n d_s(x,y)
\qquad
\mbox{for all $n \geq 0$ whenever  $y \in \gamma^s(x)$.}
\end{equation}

The class of test functions we use are required to be smooth along stable
leaves only.  Let $\Fb$ denote the set of
bounded measurable functions on $\Delta$.
For $\vf \in \Fb$ and
$0<r\leq 1$, define
\[
K^r_s(\vf) = \sup_{\gamma^s \in \Gamma^s(\Delta)} K^r(\vf|_{\gamma^s})
\; \; \; \mbox{where} \; \; \;
K^r(\vf|_{\gamma^s}) = \sup_{x,y \in \gamma^s} |\vf(x)-\vf(y)|\,d_s(x,y)^{-r} .
\]
If $A$ is an $s$-subset of $\Delta$, we define
$|\vf|_{\C^r_s(A)} = \sup_{\gamma^s \subset A}
|\vf|_{\C^0(\gamma^s)} + K^r(\vf|_{\gamma^s})$
and let
$\C^r_s(A) = \{ \vf \in \Fb: |\vf|_{\C^r_s(A)} < \infty \} $.

For $h \in (\C^r_s(\Delta))'$ an element of the dual of $\crs(\Delta)$,
the transfer operator $\Lp: (\crs(\Delta))' \to (\crs(\Delta))'$ is defined by
\[
\Lp h (\vf) = h(\vf \circ F) \qquad \mbox{for each $\vf \in \C^r_s(\Delta)$} .
\]
When $h$ is a measure absolutely continuous with respect to
the reference measure $m$, we shall
call its $L^1(m)$ density $h$ as well.
Hence $h(\vf) = \int_\Delta \vf h \, dm$.
With this convention, $L^1(m) \subset (\crs(\Delta))'$ and one can
restrict $\Lp$ to $L^1(m)$.  In this case,
\[
\Lp^n h (x) = \sum_{y \in F^{-n}x} h(y) (J_m F^n(y))^{-1}
\]
for each $n \geq 0$
where $J_mF^n$ is the Jacobian of $F^n$ with respect to $m$.

Along unstable leaves, we define the metric $d_u(x,y) = \beta_0^{s(x,y)}$
for $y \in \gamma^u(x)$ and some $\beta_0 <1$ to be chosen later.
Let Lip$^u(\vf|_{\gamma})$ denote the
Lipschitz constant of a function $\vf$ along $\gamma \in \Gamma^u$ with respect to
$d_u(\cdot, \cdot)$ and
define Lip$^u(\vf) = \sup_{\gamma \in \Gamma^u(\Delta)}
\mbox{Lip}^u(\vf|_\gamma)$.
We define Lip$^u(\Delta) = \{ \vf \in \Fb : \mbox{Lip}^u(\vf) < \infty \}$.


\subsection{Definition of norms}
\label{norms}

We recall norms constructed in \cite{demers norms}
on which the transfer operator $\Lp$ has a spectral gap.  We will show here
that this spectrum is robust
under perturbations in the form of small holes in $\Delta$.
We assume throughout that $(F,\Delta)$ has exponential return times, i.e.\
there exist constants $C>0$, $\theta < 1$ such that $m(\Delta_\ell) \leq C\theta^\ell$.

Let $\pa = \{ \dlj \}$ denote the Markov partition for $F$.
For each $k \geq 0$, define $\pa^k = \bigvee_{i=0}^k F^{-i}\pa$ and let
$\pa^k_{\ell,j} = \pa^k | \dlj$.  The elements $E \in \pa^k_{\ell,j}$ are
$k$-cylinders which are $s$-subsets of $\dlj$.
For $\psi \in L^1(m)$ and $E \in \pa^k$, define
\[
\fint_E \psi \, dm = \frac{1}{m(E)} \int_E \psi \, dm.
\]
Now choose $0 < q < p \leq 1$ and fix $1 > \beta_0 > \max\{\theta, \sqrt{\alpha} \}$
where  $\alpha$ is from {\bf (P3)}.
Next, choose
$1 > \beta \geq \max\{ \beta_0^{(p-q)/p}, \alpha^q \}$.

For $h \in \mbox{Lip}^u(\Delta)$, define the {\em weak norm} of $h$ by
$|h|_w = \sup_{\ell,j,k} |h|_{w(\pklj)}$ where
\begin{equation}
\label{eq:weak}
|h|_{w(\pklj)} = \beta_0^\ell  \sup_{E \in \pa^k_{\ell,j}} \sup_{|\vf|_{\C^p_s(E)} \leq 1}
           \fint_E h \, \vf \, dm .
\end{equation}

Define the {\em strong stable norm} of $h$ by $\|h\|_s = \sup_{\ell,j,k} \|h\|_{s(\pklj)}$ where
\begin{equation}
\label{eq:s-stable}
\|h\|_{s(\pklj)} = \beta^\ell \sup_{E \in \pa^k_{\ell,j}} \sup_{|\vf|_{\C^q_s(E)} \leq 1}
           \fint_E h \, \vf \, dm .
\end{equation}

For $\vf \in \C^p_s(\Delta)$, define $\vf_E$ on $E \in \pklj$ by
$\vf_E(x) = m(E)^{-1} \int_{\gamma^u(x) \cap E} \vf \, dm_\gamma$, for $x \in E$.
Let $\tvf_E(x) = \vf_E(\gamma^u(x))$ for $x \in \dlj$ be the extension of $\vf_E$
to $\dlj$.  Note that $\tvf_E$ is well-defined since $\vf_E$ is constant on
unstable leaves.  In what follows, let $E_k \in \pklj$, $E_r \in \pa^r_{\ell,j}$ for $r \geq k$.

We define the {\em strong unstable norm} of $h$ by
$\|h\|_u = \sup_{\ell,j,k} \|h\|_{u(\pklj)}$ where
\begin{equation}
\label{eq:s-unstable}
\|h\|_{u(\pklj)} = \sup_{E_k \in \pa^k_{\ell,j}} \sup_{E_r \subset E_k}
        \sup_{|\vf|_{\C^p_s(E_r)} \leq 1}
           \beta^{\ell -k} \left| \fint_{E_r} h \,  \vf  \, dm
           - \fint_{E_k} h \, \tvf_{E_r} \, dm \right|  .
\end{equation}

The {\em strong norm} of $h$ is defined as $\|h\| = \|h\|_s + b \|h\|_u$, for
some $b>0$ to be chosen later.

We denote by $\B$ the completion of Lip$^u(\Delta)$ in the $\|\cdot \|$-norm and by
$\B_w$ the completion of Lip$^u(\Delta)$ in the $|\cdot |_w$ norm.


\subsection{Known spectral picture for $\Lp: \B \circlearrowleft$}

The following proposition is \cite[Proposition 1.3]{demers norms}

\begin{proposition}\cite{demers norms}
\label{prop:ly}
Suppose $(F,\Delta)$ satisfies properties {\bf (P1)}-{\bf (P5)} and has exponential
return times.  Then
there exists $\bar{C}>0$ such that for each $h \in \B$ and $n \geq 0$,
\begin{eqnarray}
|\Lp^nh|_w & \leq & \bar{C} |h|_w    \label{eq:weak norm est} \\
\|\Lp^nh\|_s & \leq & \bar{C} \beta^n \|h\|_s + \bar{C} |h|_w   \label{eq:stable norm est} \\
\|\Lp^nh\|_u & \leq & \bar{C} \beta^n \|h\|_u + \bar{C} \|h\|_s  \label{eq:unstable norm est}
\end{eqnarray}
\end{proposition}

For any $1>\tau> \beta$, there exists
$N \geq 0$ such that $2\bar{C} \beta^N < \tau^N$.  Choose $b = \beta^N$.  Then,
\[
\| \Lp^N h\| =  \| \Lp^N h\|_s + b \| \Lp^N h\|_u
        \leq  \bar{C} \beta^N (\|h\|_s + b \|h\|_u) + b\bar{C} \|h\|_s + \bar{C} |h|_w
                  \leq \tau^N \|h\| + \bar{C} |h|_w .
\]
The above represents the
traditional Lasota-Yorke inequality.  By \cite[Lemma~2.6]{demers norms},
the unit ball of $\B$ is relatively compact in $\B_w$, so it follows from standard
arguments that the essential spectral radius
of $\Lp$ on $\B$ is bounded by $\beta$ (see e.g. \cite{baladi, hennion}).

\begin{thm}\cite[Theorems 1, 2]{demers norms}
\label{thm:spectral}
The operator $\Lp: \B \circlearrowleft$ is quasi-compact with essential spectral
radius bounded by $\beta$.  If $F$ is mixing, then $\Lp$ has a spectral gap and
there is a unique invariant probability measure $\tnu_{\mbox{\tiny SRB}} \in \B$ with Lipschitz densities
on $\gamma \in \Gamma^u(\Delta)$.  In addition, $\tnu_{\mbox{\tiny SRB}}$ projects to the SRB measure
for $f$, i.e.\ $\pi_*\tnu_{\mbox{\tiny SRB}} = \nusrb$.
\end{thm}


\subsection{Proof of Theorem~\ref{thm:continuity}}
\label{proof of thm continuity}

Our strategy is to prove that the spectral gap for the transfer operator persists
after the introduction of a small hole.  We adopt the perturbative approach
presented in \cite{keller liverani}.

Recall the definition of a Young tower respecting a hole $H$
from Section~\ref{results}.A (i.e.\ conditions {\bf (H.1)} and {\bf (H.2)}).
Define $\mu^u(\Lambda):= \inf_{\gamma \in
\Gamma^u} \mu_\gamma(\Lambda \cap \gamma)$.
For a
generalized horseshoe $(\Lambda, R)$ respecting a hole
$H$ or a pair of holes $H_1, H_2$, we define
\[
\begin{split}
& n(\Lambda, R; H)= \sup \{n \in {\mathbb Z}^+: \
{\rm no \ point \ in} \ \Lambda \ {\rm falls \ into} \ H \ {\rm in \
the\ first} \ n \ \rm iterates\} \ \mbox{and}    \\
& n(\Lambda, R; H_1, H_2) = \sup \{ n \in \mathbb{Z}^+ :
\mbox{no point in $\Lambda$ falls into $H_1 \triangle H_2$ in the
first $n$ iterates} \}
\end{split}
\]
where $H_1 \triangle H_2$ represents the symmetric difference between
$H_1$ and $H_2$.

Let $\{ H_t \}_{t \in I}$ be a well-parameterized sequence of holes as defined in
Section~\ref{results}.A.
For $\ve>0$, we define $\mathbb{H}_\ve = \{ H_t : t \leq \ve \}$.
The following uniform constants property is assumed for the tower construction.

\begin{assumption}
Let $\{ H_t \}_{t \in I}$ be as above.
There exist constants $C_2, \kappa>0$ and $\theta \in (0,1)$
such that for all small enough $\ve>0$, we have
the following:
\begin{itemize}
\item[(a)] For each pair $\sigma = (H_1, H_2) \in \mathbb{H}_\ve \times \mathbb{H}_\ve$,
\begin{itemize}
\vspace{-6 pt}
\item[(i)] $f$ admits a generalized horseshoe $(\Lambda^{(\sigma)},
R^{(\sigma)})$ respecting both $H_1$ and $H_2$, i.e., the combined
boundaries $\partial H_1 \cup \partial H_2$;
\item[(ii)] $(\Lambda^{(\sigma)}, R^{(\sigma)})$ is mixing.
\end{itemize}

\item[(b)] Each generalized horseshoe $(\Lambda^{(\sigma)},
R^{(\sigma)})$ from above can be constructed to have the following uniform properties:
\begin{itemize}
\vspace{-6 pt}
\item[(i)] $\Lambda^{(\sigma_1)} \approx \Lambda^{(\sigma_2)}$,\footnote{By
$\Lambda^{(\sigma_1)} \approx \Lambda^{(\sigma_2)}$,
we only wish to convey that both
horseshoes are located in roughly the same region of the
manifold $M$ and not anything technical in the sense of
convergence.} for all $\sigma_1, \sigma_2 \in \mathbb{H}_\ve \times \mathbb{H}_\ve$

\item[(ii)] $\mu^u(\Lambda^{(\sigma)}) \ge \kappa$ and
$\mu_\gamma\{R^{(\sigma)}>n\} < C_2 \theta^n$ for all $n \ge 0$;

\item[(iii)] {\bf (P3)--(P5)} hold with the constants $C_0$ and $\alpha$.
\end{itemize}
\end{itemize}
\end{assumption}

\begin{remark}
The uniformity condition {\bf (U)} has been proved for the billiard map associated
with the periodic Lorentz Gas with small holes (see \cite[Proposition 2.2 and Section 3.1]{dwy}).
For the purposes of verifying the uniformity condition, an essential feature of the sequences
of holes considered there is that the boundaries of the holes are transverse to the stable and
unstable manifolds.
Although it may at first seem a strong requirement, checking the
items listed in {\bf (U)} requires only a small modification of the tower construction
necessary to build a tower for a single hole.
\end{remark}

For the remainder of the proof, we fix $\ve >0$ and assume that
$\mathbb{H}_\ve$ satisfies the uniformity conditions {\bf (U)}.
Given $H_1, H_2 \in \mathbb{H}_\ve$,
{\bf (U)}(a)(i) immediately yields a tower $(F, \Delta)$
in which $\tH_1 = \pi^{-1}H_1$ and $\tH_2 = \pi^{-1}H_2$ are both countable
unions of Markov partition elements $\dlj$.  There are 3 towers we wish to compare:
$(F, \Delta)$ which has cuts respecting both $\tH_1$ and $\tH_2$, but no holes
have been removed; $(\F_1, \Delta(\tH_1))$, the open system corresponding to
the removal of $\tH_1$ from $\Delta$; and $(\F_2, \Delta(\tH_2))$, the open system
corresponding to $\tH_2$.
Note that $\F_1$ and $\F_2$ are both restrictions of the same map $F$.

We denote the transfer operators associated
with these systems by $\Lp$, $\Lp_1$ and $\Lp_2$ respectively.
Since $\Delta(\tH_1), \Delta(\tH_2) \subset \Delta$, we may consider
all three operators acting on a single Banach space $(\B, \| \cdot \|)$ as defined in
Section~\ref{norms}.

We let $\Delta^n(\tH_i)$ denote the set of points in $\Delta$ which have not escaped
from the tower with hole $\tH_i$ by time $n$.
Notice that for $h \in \mbox{Lip}^u(\Delta)$,
\begin{equation}
\label{eq:hole op}
\Lp_i h = 1_{\Delta \backslash \tH_i} \Lp ( 1_{\Delta \backslash \tH_i} h )
= \Lp (1_{\Delta^1(\tH_i)} h) \; \; \mbox{for $i = 1,2$.}
\end{equation}
Since $1_{\tH_i} \in \B$, it follows that $\Lp_1$ and $\Lp_2$ satisfy the
inequalities \eqref{eq:weak norm est}-\eqref{eq:unstable norm est} from
Proposition~\ref{prop:ly} with uniform constants
$\bar{C} >0$, $\beta<1$.

In order to carry out the perturbation argument in \cite{keller liverani}, we introduce
the following norm for operators $\mathcal{K} : \B \to \B_w$:
\[
||| \mathcal{K} ||| = \sup \{ |\mathcal{K} h |_w : \| h \| \leq 1 \} .
\]

\begin{lemma}
\label{lem:close L}
There exists $C>0$ such that for any $H_1, H_2 \in \mathbb{H}_\ve$,
\begin{eqnarray}
||| \Lp - \Lp_i ||| & \leq & C (\beta^{-1} \beta_0)^{n(\Lambda, R; H_i)} \; \mbox{for $i =1,2$, and}
	\label{eq:1 hole close} \\
||| \Lp_1 - \Lp_2 ||| & \leq & C (\beta^{-1} \beta_0)^{n(\Lambda, R; H_1, H_2)}.
	\label{eq:2 hole close}
\end{eqnarray}
\end{lemma}

\begin{proof}
By the density of Lip$^u(\Delta)$ in $\B$ and $\B_w$, it suffices to derive these inequalities for $h \in \Lip^u(\Delta)$.
Now let $h \in \mbox{Lip}^u(\Delta)$, $\|h\| \leq 1$. Fix $E \in \pklj$ and
take $\vf \in \cps(E)$ with $|\vf|_{\cps(E)} \leq 1$.

We first consider the case when
$\ell=0$ and $E \in \pa^k_{0,j}$.
Note that $F^{-1}E$ is comprised
of a countable union of $(k+1)$-cylinders, $F^{-1}E = \cup E'$,
$E' \in \pa^{k+1}_{\ell',j'}$.
On level $\ell' \le n(\Lambda, R; \tH_1) $, we have $1_{\Delta^1(\tH_1)} = 1$.
Also, since $\Delta^1(\tH_1)$ is a union of 1-cylinders, we have either
$1_{\Delta^1(\tH_1)}|_{E'} \equiv 0$ or $1_{\Delta^1(\tH_1)}|_{E'} \equiv 1$
for each $E' \subset F^{-1}E$ when $\ell'  > n(\Lambda, R; \tH_1)$.
Thus by \eqref{eq:hole op},
\begin{equation}
\label{eq:zero}
\int_E (\Lp - \Lp_1) h \, \vf \, dm
= \sum_{E'} \int_{E'}(1-1_{\Delta^1(\tH_1)})h \, \vf\circ T \, dm
\leq \sum_{\ell' \geq n(\Lambda, R;H_1) } \! \! \! \beta^{-\ell'} m(E') \|h\|_s |\vf\circ F^n|_{\cqs(E')}
\end{equation}
where $E' \subseteq \Delta_{\ell',j'}$.
To estimate $|\vf\circ F^n|_{\cqs(E')}$,
take $x,y \in \gamma^s \subset E'$ and write
\[
|\vf \circ F^n(x) - \vf \circ F^n(y)| \; \leq \; K^q_s(\vf) d_s(F^nx, F^ny)^q
\; \leq \; K^q_s(\vf) C_0 \alpha^{qn} d_s(x,y)^q
\]
by \eqref{eq:ds contract} since $q< p$.
This together with $|\vf \circ F^n|_\infty = |\vf|_\infty$ implies
$|\vf\circ F^n|_{\cqs(E')} \leq C_0 |\vf|_{\cqs(E)} $.

Due to bounded distortion given by Property {\bf (P4)}(b) and Lemma~\ref{lem:jacobian},
we have
$\frac{m(E')}{m(E)}
\leq C_1 \frac{m(E'_1)}{m(\Delta_0)}$
where $E'_1$ is the $1$-cylinder containing $E'$.  Thus by
{\bf (U)}(b)(ii),
there exists $C>0$ such that,
\begin{equation}
\label{eq:distortion}
m(E') \leq C m(E'_1) m(E).
\end{equation}

Now \eqref{eq:zero} becomes,
\[
\int_E (\Lp - \Lp_1) h \, \vf \, dm
\leq CC_0 \sum_{\ell' \geq n(\Lambda, R;H_1) } \beta^{-\ell'} m(E'_1) m(E) \|h\|
\leq C' (\beta^{-1} \theta)^{n(\Lambda, R; H_1)} m(E) \|h\|
\]
since $\theta < \beta$.  Dividing by $m(E)$ and
taking the supremum over $\vf \in \cps(E)$ and $E\in \pa^k_{0,j}$, we have
\begin{equation}
\label{eq:weak zero}
|(\Lp - \Lp_1) h|_{w(\pa^k_{0,j})} \leq C' (\beta^{-1} \theta)^{n(\Lambda, R; H_1)} \|h\| .
\end{equation}

Next consider $1 \leq \ell \le n(\Lambda, R; H_1)$.
Then \eqref{eq:hole op} implies that $(\Lp h)|_{\Delta_\ell} = (\Lp_1 h) |_{\Delta_\ell}$,
so
\[
\int_E (\Lp - \Lp_1)h \, \vf \, dm = 0 .
\]

Finally, we consider the case $\ell > n(\Lambda, R; H_1)$, then
\[
\int_E (\Lp - \Lp_1)h \, \vf \, dm  =  \int_{F^{-1}E} (1 - 1_{\Delta^1(\tH_1)}) h \,
\vf \circ T \, dm
 \le \| h\|_s \beta^{-\ell+1}m(E) |\vf|_{\cqs(E)} ,
\]
since $m(F^{-1}E) = m(E)$.
Taking the appropriate suprema in the definition of the weak norm,
we obtain for $\ell \geq 1$,
\begin{equation}
\label{eq:one}
|(\Lp - \Lp_1) h|_{w(\pklj)} \leq (\beta^{-1}\beta_0)^{n(\Lambda,R;H_1)} \|h\| .
\end{equation}
Combining \eqref{eq:weak zero} and \eqref{eq:one} proves
\eqref{eq:1 hole close} for $i=1$ since $\theta < \beta_0 < \beta$.
The proof for $i=2$ is identical.

To prove \eqref{eq:2 hole close}, replace $\Lp$ by $\Lp_2$ and
the analogous estimates follow using the fact that
$1_{\Delta^1(\tH_1)} = 1_{\Delta^1(\tH_2)}$ on all levels
$\ell \le n(\Lambda, R; H_1, H_2) $.
\end{proof}

Let $H_1 = H_{t_1}$, $H_2=H_{t_2} \in \mathbb{H}_\ve$
(we allow the possibility that one of the holes is the infinitesimal
hole $H_0$).

By definition, dist$(H_1, H_2) \leq |t_1 - t_2|$.  By condition {\bf (H.2)},
$d(f^\ell \Lambda, \Si_{H_1} \cup \Si_{H_2}) \geq \delta \xi_1^{-\ell}$ so that
$n(\Lambda, R; H_1, H_2) \geq - \log(|t_1 - t_2|/\delta)/\log \xi_1$ for $|t_1 - t_2| < \delta$.
Lemma~\ref{lem:close L} implies that
\[
\begin{split}
||| \Lp - \Lp_i ||| & \leq C \delta^{-1} |t_i|^{\log(\beta_0^{-1} \beta)/\log \xi_1}
\; \mbox{for $i =1,2$, and}  \\
||| \Lp_1 - \Lp_2 ||| & \leq  C \delta^{-1} |t_1 - t_2|^{\log(\beta_0^{-1} \beta)/\log \xi_1}.
\end{split}
\]
Now the results of \cite{keller liverani} imply that both the spectra and spectral projectors outside any
disk of radius greater than $\beta$ vary H\"older continuously in the size of the
perturbation.  Since the original dynamics are mixing
by {\bf (U)}(a)(ii), $\Lp$ has a spectral gap
by Theorem~\ref{thm:spectral}.
Thus there exists $\delta >0$ such that the spectral gap for $\Lp$ is preserved for $\Lp_i$ for
$t_i \le \delta$.

Let $\tmu_i$ denote the physical quasi-invariant measure in $\B$ corresponding to the
leading eigenvalue $\ra_i$ of $\Lp_i$, $i=1,2$.  Then \cite[Theorem 4.4]{dwy}
implies that the escape rate from $\Delta(\tH_i)$ with respect to
$\tmu_i$, $-\rho(\tmu_i)$, exists and equals $-\log \ra_i$.  By
\cite[Theorem 2]{dwy}, the escape rate $-\rho_i(\nusrb)$ from
$M\backslash H_i$ with respect to $\nusrb$ equals $-\log \ra_i$ as well.

In particular, we have
$|\ra_1 - \ra_2| \leq C' |t_1 - t_2|^{\bar \alpha}$ for any
$\bar \alpha < \log(\beta_0^{-1} \beta)/\log \xi_1$.
This implies that $t \mapsto \rho(t)$ is a H\"older continuous function
for $t \leq \delta$.
In addition, \cite{keller liverani} also implies that the spectral projectors vary
H\"older continuously so that the quasi-invariant measures
$\tmu_t$ vary H\"older continuously in the weak norm $| \cdot |_w$ for $t \leq \delta$.
Projecting these quasi-invariant measures to $M$, we obtain measures
$\mu_t = \pi_*\tmu_t$, which are quasi-invariant with respect to $f|_{M\backslash H_t}$.
It follows from \cite[Lemmas 2.2 and 4.1]{demers norms} that the weak norm dominates the integral, i.e.,
\[
|\tmu_t(\vf) | \le C |\tmu_t|_w |\vf|_{L^\infty}
\]
for all $\vf \in \mathcal{F}_b$ which are continuous on each element $\dlj$.
Since a bounded continuous function $\vf$ on $M$
lifts to a bounded function $\vf \circ \pi$ on $\Delta$ that is continuous on each $\dlj$,
the projected measures $\mu_t$
also vary H\"older continuously with $t$:  Given $\vf \in C^0(M)$,
we have $|\mu_{t_1}(\vf) - \mu_{t_2}(\vf) | \leq C''|t_1 - t_2|^{\bar \alpha} |\vf|_{C^0(M)}$ for
$t_1, t_2 \le \delta$.



\begin{thebibliography}{999999}

\bibitem[B]{baladi} V. Baladi, {\em  Positive transfer operators and decay
    of correlations}, Advanced Series in Nonlinear Dynamics, {\bf 16},
    World Scientific (2000).

\bibitem[Bo1]{bowen entropy} R. Bowen, \emph{Entropy-expansive maps}, Trans. Amer. Math. Soc. {\bf 64} (1972), 323-331.

\bibitem[Bo2]{bowen book}  R. Bowen, \emph{Equilibrium states and the ergodic theory
    of Anosov diffeomorphisms}.
    Lecture Notes in Math. {\bf 470}.  Springer-Verlag: Berlin, 1975.

\bibitem[BDM]{bdm}  H. Bruin, M.F. Demers and I. Melbourne, \emph{Existence and convergence properties of physical measures for certain dynamical systems with holes},
to appear in Ergod. Th. and Dynam. Systems.

\bibitem[BY]{buni yur}  L.A. Bunimovich and Alex Yurchenko, \emph{Where to place a
hole to achieve a maximal escape rate}, Israel J. of Math. {\bf 182} (2011), 229-252.

  \bibitem[Ce]{cencova1}  N. N. \v{C}encova, \emph{A natural invariant measure on Smale's horseshoe}, Soviet Math. Dokl.
    {\bf 23} (1981), 87-91.

\bibitem[CM1]{chernov mark1} N. Chernov and R. Markarian, \emph{Ergodic
    properties of Anosov maps with rectangular holes}, Bol. Soc. Bras. Mat.
    {\bf 28} (1997), 271-314.

\bibitem[CM2]{chernov mark2} N. Chernov and R. Markarian, \emph{Anosov maps
    with rectangular holes.  Nonergodic cases}, Bol. Soc. Bras. Mat.
    {\bf 28} (1997), 315-342.

\bibitem[CMT1]{chernov mt1} N. Chernov, R. Markarian and S. Troubetzkoy,
    \emph{Conditionally invariant measures for Anosov maps with small holes},
    Ergod. Th. and Dynam. Sys. {\bf 18} (1998), 1049-1073.

\bibitem[CMT2]{chernov mt2} N. Chernov, R. Markarian and S. Troubetzkoy,
    \emph{Invariant measures for Anosov maps with small holes},
    Ergod. Th. and Dynam. Sys. {\bf 20} (2000), 1007-1044.

  \bibitem[CV]{chernov bedem}  N. Cehrnov and H. van den Bedem, \emph{Expanding maps of an interval with holes},
    Ergod. Th. and Dynam. Sys. {\bf 22} (2002), 637-654.

   \bibitem[CMS1]{collet ms1}  P. Collet, S. Mart\'{i}nez and B. Schmitt, \emph{The Yorke-Pianigiani measure and
    the asymptotic law on the limit Cantor set of expanding systems}, Nonlinearity {\bf 7} (1994), 1437-1443.

 \bibitem[D1]{demers tower hole} M.F. Demers, \emph{Markov extensions for dynamical systems
 with holes:  An application to expanding maps of the interval}, Israel J. of Math. {\bf 146} (2005), 189-221.

    \bibitem[D2]{demers norms}  M.F. Demers, \emph{Functional norms for Young towers},
Ergod. Th. and Dynam. Sys. {\bf 30}:5 (2010), 1371-1398.

\bibitem[DL]{demers liverani}  M.F. Demers and C. Liverani, \emph{Stability of statistical properties
in two-dimensional piecewise hyperbolic maps}, Trans. Amer. Math. Soc. {\bf 360}:9
(2008), 4777-4814.

\bibitem[DWY1]{dwy}  M.F. Demers, P. Wright and L.-S. Young, \emph{Escape rates and
physically relevant measures for billiards with small holes}, Commun. Math. Physics
{\bf 294}:2 (2010), 353-388.

\bibitem[DWY2]{dwy2} M.F. Demers, P. Wright and L.-S. Young, \emph{Entropy, Lyapunov exponents and escape rates for open systems}, to appear in Ergodic Theory and Dynamical Systems.

\bibitem[DY]{demers young} M.F. Demers and L.-S. Young, \emph{Escape rates and
conditionally invariant measures}, Nonlinearity {\bf 19} (2006), 377-397.

\bibitem[FP]{ferg polli}  A. Ferguson and M. Pollicott, \emph{Escape rates for Gibbs measures},
to appear in Ergod. Th. and Dynam. Sys.

\bibitem[HH]{hennion} H. Hennion and L. Herv\'e, \emph{Limit Theorems for
  Markov Chains and Stochastic Properties of Dynamical Systems via
  Quasi-Compactness}, Lecture Notes in Mathematics {\bf 1766}, Springer: Berlin, 2001.

\bibitem[KL1]{keller liverani} G. Keller and C. Liverani, \emph{Stability of the spectrum for
    transfer operators}, Annali della Scuola Normale Superiore di Pisa,
    Scienze Fisiche e Matematiche, (4) {\bf XXVIII} (1999), 141-152.

\bibitem[KL2]{kl escape} G.\ Keller and C.\ Liverani, \emph{Rare events, escape rates and
quasi-stationarity: some exact formulae}, J. Stat. Phys. {\bf 135}:3 (2009), 519-534.

 \bibitem[LiM]{liverani maume}  C. Liverani and V. Maume-Deschamps, \emph{Lasota-Yorke maps with holes: conditionally
    invariant probability measures and invariant probability measures on the survivor set}, Annales de l'Institut Henri
    Poincar\'{e} Probability and Statistics, {\bf 39} (2003), 385-412.

 \bibitem[PY]{pianigiani yorke}  G. Pianigiani and J. Yorke,  \emph{Expanding maps on sets which are almost invariant:
    decay and chaos}, Trans. Amer. Math. Soc. {\bf 252} (1979), 351-366.

 \bibitem[W]{walters} P. Walters, \emph{An Introduction to Ergodic Theory}, Graduate Texts in Mathematics {\bf 79}, Springer-Verlag:New York, 1981.

  \bibitem[Y1]{young escape}  L.-S. Young, \emph{Some large deviation results for dynamical systems},
    Trans. Amer. Math. Soc. {\bf 318} (1990), 525-543.

  \bibitem[Y2]{young tower}  L.-S. Young, \emph{Statistical properties of
    dynamical systems with some hyperbolicity}, Annals of Math.
    {\bf 147}:3 (1998), 585-650.

\end{thebibliography}
\end{document}